\documentclass[12 pt]{amsart}

\usepackage[margin=2.3cm]{geometry} 
\usepackage{graphicx}
\usepackage{amssymb,latexsym}
\usepackage{amsmath,amsfonts,amstext,bbm}
\usepackage[utf8]{inputenc}
\usepackage{pst-plot}
\usepackage{pst-infixplot}
\usepackage{tikz}
\usetikzlibrary{calc}
\usepackage{pgfplots}

\newtheorem{theorem}{Theorem}

\newtheorem{proposition}{Proposition}

\newtheorem{corollary}{Corollary}

\DeclareMathOperator{\M}{\mathcal{M}}
\DeclareMathOperator{\A}{\mathcal{A}}
\DeclareMathOperator{\N}{\mathbb{N}}
\DeclareMathOperator{\Z}{\mathbb{Z}}
\DeclareMathOperator{\R}{\mathbb{R}}
\DeclareMathOperator{\C}{\mathbb{C}}
\DeclareMathOperator{\Ma}{\bf M}
\DeclareMathOperator{\Aa}{\textbf{A}}

\DeclareMathOperator{\Tt}{\mathbb{T}}

\begin{document}

	\title[The higher regularity of  the maximal function]
{The higher regularity of  the discrete Hardy-Littlewood maximal function}

\author{Faruk Temur, Hİkmet Burak Özcan}
\address{Department of Mathematics\\
	Izmir Institute of Technology, Urla, Izmir, 35430,
	Turkey}
\email{faruktemur@iyte.edu.tr, hikmetozcan@iyte.edu.tr }
\date{April 6, 2025}

\begin{abstract}
	In a  recent  note the first author   gave  the first positive result on the  second order regularity of the  discrete noncentered Hardy-Littlewood maximal function. In this article we conduct a thorough investigation of possible similar results for higher order derivatives. We uncover that such results are indeed a consequence of a stronger phenomenon regarding the growth of $l^p(\Z)$ norms of the derivatives of characteristic functions of finite subsets of $\Z$. Along the way we discover very interesting connections to Remez-type inequalities for exponential sums,   zeros of complex polynomials with restricted coefficients (Littlewood-type polynomials), and to subset sum problem of computer science.
\end{abstract}

\maketitle

\section{Introduction}\label{intro}

Regularity of the Hardy-Littlewood maximal function has become a focus of intense interest in the last 40 years,  see the survey article \cite{ct}. Despite the vast literature  developed on this problem, almost all  efforts  concentrated on the first order regularity, mainly because higher order regularity results of the same generality are not possible, see \cite{ko,jw}. However, 
in a recent article \cite{tem3} the first author obtained the following result, which was the first positive result for higher order regularity, and demonstrated that positive results are possible for more restricted classes of functions: 
\begin{equation}\label{pos1}
	\|(\Ma^d \chi_A)''  \|_{l^p(\Z)}  \leq2^{1-\frac{1}{p}} 3^{\frac{1}{p}} \|\chi_A''\|_{l^p(\Z)},  
\end{equation}
for any finite subset $A\subset \Z$ and its characteristic function $\chi_A$, and $1\leq p\leq \infty$. Here, $\Ma^d f$ is the discrete noncentered Hardy-Littlewood maximal function  of $f:\mathbb{Z} \rightarrow \mathbb{R}$ defined by  	
\begin{equation*}
	\Ma^df(n):=\sup_{r,s\in \Z^+} \Aa_{r,s}|f|(n),  \quad   \text{where}   \quad 	\Aa^d_{r,s}f(n)=\frac{1}{r+s+1}\sum_{j=-r}^{s}f(n+j),
\end{equation*}
with $\Z^+$ denoting the nonnegative integers. In this  discrete context regularity is studied using discrete derivatives  -also called forward differences- defined by 
\begin{equation*}
	\begin{aligned}
		f^{(k)}(n)&:= \sum_{j=0}^{k}{k \choose j}(-1)^{j}f(n+k-j), \qquad k\geq 1.
	\end{aligned}
\end{equation*}

But  it
  is possible to prove \eqref{pos1} in a more general way. We have
\begin{theorem}
	Let $A$ be a finite subset of integers, and $\chi_A$ be its characteristic function.  Then for any $k\geq 1$ and $1\leq p\leq \infty$
	\begin{equation*}\label{pos-1}
		\|\chi_A^{(k)}\|_{l^p(\Z)}  \geq C_{k,p} \|\chi_A'\|_{l^p(\Z)}.
	\end{equation*}
\end{theorem}
Then as a corollary we have
\begin{corollary}
	Let $A$ be a finite subset of integers, and $\chi_A$ be its characteristic function.  Then for any $k\geq 1$ and $1\leq p\leq \infty$ we have 
	\begin{equation*}\label{}
		\|(\Ma^d \chi_A)^{(k)} \|_{l^p(\Z)}   \lesssim_{k,p} \|\chi_A^{(k)}\|_{l^p(\Z)} .  
	\end{equation*}
	\end{corollary}
Here the notation $A\lesssim_{a,b,c} B$ means $A\leq C_{a,b,c}B$, and we will also use the analogous notation $A\gtrsim_{a,b,c}B$. This corollary  is immediately obtained even for  an arbitrary operator $T$: if we know
\begin{equation}\label{pos41}
	\|(T\chi_A)'  \|_{l^p(\Z)}   \lesssim_p \|\chi_A'\|_{l^p(\Z)} , 
\end{equation}
 then by our Theorem 1  
\begin{equation*}\label{pos51}
	\|(T \chi_A)^{(k)} \|_{l^p(\Z)}   \leq   2^{k-1}\|(T \chi_A)' \|_{l^p(\Z)}  \lesssim_{k,p}   \|\chi_A'\|_{l^p(\Z)} \lesssim_{k,p}       \|\chi_A^{(k)}\|_{l^p(\Z)} .  
\end{equation*}
For $\Ma^d$,  \eqref{pos41} follows from  \cite{bchp} and properties of characteristic functions. For the centered discrete Hardy-Littlewood maximal function 
\begin{equation*}
	\M^d f(x):=\sup_{r\in \Z^+} \A^d_r|f|(x),  \quad   \text{where}   \quad 	\A_{r}^df(n):=\frac{1}{2r+1}\sum_{j=-r}^{r}f(n+j), 
\end{equation*} 
  we know  \eqref{pos41} by \cite{tem}  and properties of characteristic functions. 
Therefore, 
\begin{corollary}
	Let $A$ be a finite subset of integers, and $\chi_A$ be its characteristic function.  Then for any $k\geq 1$ and $1\leq p\leq \infty$ we have 
	\begin{equation*}\label{}
		\|(\M^d \chi_A)^{(k)} \|_{l^p(\Z)}   \lesssim_{k,p} \|\chi_A^{(k)}\|_{l^p(\Z)}.  
	\end{equation*}
\end{corollary}

For small values of $k$, it is possible to obtain Theorem 1 via a simple observation on binomial coefficients. This argument works up to $k\leq 7$, but becomes unavailable thereafter.

We give two short proofs of Theorem 1, both relying on the  fact that if $\chi_A^{(k)}$ is zero for an interval of certain length, this will impose conditions on $\chi_A$ that it cannot satisfy as a characteristic function of a finite set. The first proof gives a constant $C_{k,p}=2^{-k/p}$, while the second gives a much better $(2k+1)^{-1/p}.$ 
A basic observation regarding endpoints of $A$  gives the exponential but nonuniform bound:

\begin{theorem}
	Let $A\subset \Z$ be a finite nonempty set. Let $1\leq p\leq \infty$. Let $n=\lfloor k/3\rfloor$. 	We have
	\begin{equation*}
		\|\chi_A^{(k)}\|_{l^p(\Z)}  \geq \frac{1}{3}{k \choose n}.
	\end{equation*}
\end{theorem}
 Stirling's formula gives an approximation like  $2^{0.918k}$ or $(1.88)^k$ for ${k \choose k/3}.$ So this is a good exponential bound, but as there is no term $	\|\chi_A'\|_p$ on the right hand side, it is not uniform.

 Direct calculations of the constant for simple sets lead us to believe that a uniform bound $C_{k,p}\geq2^{k-o_p(k)}$ may be possible. We will present these calculations in Section 2. Furthermore  the Fourier transform allows us to prove 
such bounds for fixed sets $A$. Applying the Nazarov-Tur\'{a}n bound of \cite{naz1} after using the Fourier transform we obtain

\begin{theorem}
	Let $A\subset \Z$ be a fixed finite and nonempty set,  and  $1\leq p\leq \infty$ be a fixed number. Then for $k$ large enough depending on $|A|$
	\begin{equation}\label{nb1}
	\|\chi_A^{(k)}\|_{l^p(\Z)}  \geq 	2^{k-1-\frac{|A|-1}{2}\log_2 2k}\Big(\frac{\sqrt{|A|-1}}{7\pi e^{3/2} } \Big)^{|A|-1}(k+1)^{-\frac{1}{p'}}|A|^{\frac{1}{p}},
	\end{equation}
where $p'$ is the dual exponent of $p$.	Also, for $k$ large enough depending on the size of the left boundary $|\partial_l A|$ of $A$,
	\begin{equation}\label{nb2}
		\|\chi_A^{(k+1)}\|_{l^p(\Z)}  \geq 	2^{k-1-\frac{2|\partial_l A|-1}{2}\log_2 2k}\Big(\frac{\sqrt{2|\partial_l A|-1}}{7\pi e^{3/2} } \Big)^{2|\partial_lA|-1}(k+1)^{-\frac{1}{p'}}  \big(2|\partial_{l}A|\big)^{\frac{1}{p}}.
	\end{equation}

\end{theorem}
 For the definition of the left boundary see subsection 3.1.  Applying a bound of Borwein and Erd\'{e}lyi \cite{be} instead of the Nazarov-Tur\'{a}n bound we get 

\begin{theorem}
	Let $A\subset \Z$ be a fixed finite and nonempty set,  and  $1\leq p\leq \infty$ be a fixed number. Then  for $k$ large enough depending on $|A|$
	\begin{equation}\label{be1}
		\|\chi_A^{(k)}\|_{l^p(\Z)}  \geq 2^{k-1-3(\log_2e)\big(\frac{c\pi}{4}\big)^{\frac{2}{3}}k^{\frac{1}{3}}}\big[(k+1)|A|\big]^{-\frac{1}{p'}},
	\end{equation}
where $c$ is an absolute constant coming from the Borwein-Erd\'{e}lyi bound.
\end{theorem}

By the observation of Borwein, Erd\'{e}lyi, K\'{o}s \cite{bek} on zeros of polynomials  there is a least nonnegative integer $0\leq a\leq \log |A|$ depending on the set $A$ for which $\widehat{\chi_A}^{(a)}(1/2)\neq 0.$ We will explain this in detail in Section 4. Utilizing this fact we get

\begin{theorem}
	Let $A\subset \Z$ be a fixed finite and nonempty set,  and  $1\leq p\leq \infty$ be a fixed number.  Then for       $k$ large enough depending on $|A|$,
	\begin{equation}\label{bek1}
		\|\chi_A^{(k)}\|_{l^p(\Z)}  \geq 2^{k-\frac{a}{2}\log_2 2k+a-2}\Big[\frac{a}{e}\Big]^{\frac{a}{2}}\frac{1}{a!}\big[(k+1)|A|\big]^{-\frac{1}{p'}}.
	\end{equation}
\end{theorem}

Proofs of these three theorems, as well as heuristic investigations on the behavior of $\|\chi_A^{(k)}\|_p $ as $k\rightarrow \infty$ will be given in Section 4. We also obtain the following  result that combines  the Fourier approach with a  sparsity assumption for the sets $A.$

\begin{theorem}
	Let a finite nonempty  set $A$ satisfy the sparsity condition
	\begin{equation}\label{spc}
		\sum_{\substack{m,n\in A \\ m<n}}\frac{1}{n-m}\leq \frac{\pi}{8}|A|.
	\end{equation}
	Then for any $k\in \N$ and $p\geq 2$,
	\begin{equation*}
		\begin{aligned}
			\big\|\chi_A^{(k)}\big \|_{l^{p}(\Z)}\geq (k+1)^{\frac{1}{p}-\frac{1}{2}}2^{\frac{k}{2}-1}\big\|\chi_A\big \|_{l^{p}(\Z)},
		\end{aligned}
	\end{equation*}
and for $1\leq p< 2$,
	\begin{equation*}
	\begin{aligned}
		\big\|\chi_A^{(k)}\big \|_{l^{p}(\Z)}\geq 	 2^{k-\frac{k+2}{p}}	\big\|\chi_A\big \|_{l^{p}(\Z)}.
	\end{aligned}
\end{equation*}

\end{theorem}

The short note \cite{tem3} attracted interest of researchers to the question of higher regularity, see \cite{jw} for an elaborate counterexample showing impossibility of generalizing higher regularity to larger classes of functions.

The organization of the rest of the article is as follows. The second section contains computations and heuristics leading us to believe the possibility of the bound $C_{k,p}\geq2^{k-o_p(k)}$. The third section proves Theorem 1 and Theorem 2. The last section proves the last 4 theorems, all of which use the Fourier approach.

\section{ Calculations and heuristics via simple sets}

We start with the observation  that the support of $\chi_A^{(k)}$    lies in
\begin{equation*}
	\bigcup_{j=0}^k A-k.
\end{equation*}
Hence we can write for any $p\geq 1$
\begin{equation}\label{hol}
\|\chi_A^{(k)}\|_{l^p(\Z)}	\leq \|\chi_A^{(k)}\|_{l^1(\Z)}\leq	\|\chi_A^{(k)}\|_{l^p(\Z)}\big((k+1)|A|\big)^{\frac{1}{p'}}.
\end{equation}
Thus  once  growth rates for $p=1$ case are known, by \eqref{hol} they can be  deduced for any $p\geq 1$.  Therefore in our heuristic calculations below we will take $p=1$, and use the notation $\asymp$ to mean that the two sides are heuristically the same. We will use the approximation of binomial coefficients by normal distributions. So

\begin{equation}\label{bcn}
	{k \choose n} \asymp \frac{2^k}{\sqrt{\pi k/2}} e^{-\frac{(n-k/2)^2}{k/2}}.
\end{equation}
Thus large  binomial coefficients  are those with $|n-k/2|<k^{1/2+\varepsilon}$  for   small $\varepsilon>0.$  

For   the simplest nonempty sets $A$ that contain only one point, plainly   $\|	\chi_A^{(k)}\|_{l^1(\Z)}=2^k.$  More decay can be obtained by taking two consecutive integers. Let $A=\{0,1\}$.
\begin{equation*}
	\begin{aligned}
		\|	\chi_A^{(k)}\|_{l^p(\Z)}=\sum_{n=-1}^{k}\Big| {k  \choose n+1}-{k  \choose n}\Big|\asymp   \frac{2^k}{\sqrt{\pi k/2}} \sum_{n=-1}^{k} e^{-\frac{(n-k/2)^2}{k/2}}\Big| e^{\frac{-(n+1-k/2)^2}{k/2}+\frac{(n-k/2)^2}{k/2}} -1  \Big| ,
	\end{aligned}
\end{equation*}
where  if  $n<0$ or $n>k$ we take ${k \choose n}$
to be zero. As the significant values of $n$ are those close to $k/2$, the exponent of $e$ inside the paranthesis is small, and we can just take the first two terms of the Taylor expansion. Then we can write
\begin{equation*}
	\begin{aligned}
	e^{\frac{-(n+1-k/2)^2}{k/2}+\frac{(n-k/2)^2}{k/2}} -1=e^{\frac{-1+(k-2n)}{k/2}} -1	\asymp  \frac{-1+(k-2n)}{k/2}.
	\end{aligned} 
\end{equation*}
 The term $(k-2n)$ dominates $1$ except for  $k=2n.$ So we can write  
\begin{equation*}
	\begin{aligned}
		\sum_{n=-1}^{k}\Big| {k  \choose n+1}-{k  \choose n}\Big| \asymp \frac{2^{k+\frac{5}{2}}}{(\pi k^3)^{\frac{1}{2}}} \sum_{n=0}^{k} e^{-\frac{(n-k/2)^2}{k/2}}|k/2-n|.	
	\end{aligned}
\end{equation*}
This  can then be approximated by 
\begin{equation*}
	\begin{aligned}
		\asymp 	\frac{2^{k+\frac{7}{2}}}{(\pi k^3)^{\frac{1}{2}}}\int_0^{k/2} e^{-\frac{n^2}{k/2}}ndn\asymp	\frac{2^{k+\frac{5}{2}}}{(\pi k)^{\frac{1}{2}}}\int_0^{\infty} e^{-n^2}ndn = \frac{2^{k+\frac{3}{2}}}{(\pi k)^{\frac{1}{2}}}.
	\end{aligned}
\end{equation*}
So  cancellation leads to a factor $k^{\frac{1}{2}}$ in the denominator.

We may think of increasing  cancellation yet further by taking  $A:=\{0,1,3,4  \}.$  
\begin{equation}\label{bcn3}
	\begin{aligned}
			\|	\chi_A^{(k)}\|_{l^1(\Z)}&=\sum_{n=-4}^{k}\Big| -{k  \choose n+4}+{k  \choose n+3}+{k  \choose n+1}-{k  \choose n}\Big|
		\\     &\asymp \sum_{n=-4}^{k}\frac{2^k}{\sqrt{\pi k /2}} e^{-\frac{(n-k/2)^2}{k/2}} \Big| -e^{\frac{-4^2+4(k-2n)}{k/2}}+e^\frac{-3^2+3(k-2n)}{k/2}+e^\frac{-1^2+1(k-2n)}{k/2}-1\Big|. 
	\end{aligned}
\end{equation}
Doing a first order Taylor expansion of exponentials inside the absolute value gives
\begin{equation}\label{bcn2}
	\begin{aligned}
		&\asymp \Big| -\frac{-4^2+4(k-2n)}{k/2}+\frac{-3^2+3(k-2n)}{k/2}+\frac{-1^2+(k-2n)}{k/2}\Big|\\ &= \Big|\frac{2}{k}\big(4^2-3^2-1^2\big)+ \frac{2k-4n}{k}\big(-4+3+1\big)\Big| = \frac{12}{k}. 
	\end{aligned}
\end{equation}
Then  \eqref{bcn3} can be estimated by using \eqref{bcn}
\begin{equation*}
	\begin{aligned}
		\asymp  \frac{12}{k}2^k\frac{1 }{(\pi k/2)^{1/2}}	\sum_{n=0}^{k} e^{-\frac{(n-k/2)^2}{k/2}} \asymp
 \frac{12}{k} 2^{k}. 
	\end{aligned}
\end{equation*}
So in this way another half power of $k$ is gained.  But more importantly this calculation sheds light unto how much cancellation can be obtained in this way. Suppose we have a set    
\begin{equation*}
A:=\Big\{ a_j\in 2\Z,  \ b_j\in 2\Z+1, \ 1\leq j\leq m \ \Big| \   \sum_{j=1}^m a_j^d=	\sum_{j=1}^m b_j^d,    \quad     1\leq d\leq D     \Big\} 
\end{equation*}
of $2m$  elements, all which are  much smaller than $k.$  This then would yield a sum 
\begin{equation*}
	\begin{aligned}
		\sum_{n=0}^{k}\Big| \sum_{j=1}^m{k  \choose n+a_j}-{k  \choose n+b_j}\Big|\asymp \sum_{n=0}^{k}\frac{2^k}{\sqrt{\pi k /2}} e^{-\frac{(n-k/2)^2}{k/2}} \Big|\sum_{j=1}^m e^{\frac{-a_j^2+a_j(k-2n)}{k/2}}-\sum_{j=1}^me^\frac{-b_j^2+b_j(k-2n)}{k/2}\Big|.
	\end{aligned}
\end{equation*}
We may suppose $|k-n/2|\leq k^{1/2+\varepsilon}$, as these are the larger terms. Then the first $\lfloor D/2 \rfloor +1$ terms from the Taylor expansions of exponentials inside the absolute value     will all vanish.  From the higher order terms we  will get  $k^{\lfloor D/2 \rfloor}$ in the denominator. But then the question is how large  $D$ can be for a given $m$, or how large $m$ has to be for a given  $D$? This question is similar to  the well known  and  still unresolved Prouhet-Tarry-Escott (PTE) problem of number theory, but has the very important difference that $a_j,b_j$ are chosen from different parities. This introduces powers of $2$ to one side of the equations of the problem, and  drives the number of variables to be much larger  than the PTE problem. For the PTE problem for a given $D$ it is known that $D+1\leq m\leq D(D+1)/2,$ while for this problem we need $2^D\leq m,$ which follows from an observation of Borwein, Erd\'{e}lyi and K\'{o}s \cite{bek}.  But as we must also have these variables $\lesssim \sqrt{k}$ to  apply the Taylor expansion, we must have $m\lesssim \sqrt{k} $, implying $D\lesssim \log k$.   
In the end  we get to bounds of the form 
\begin{equation*}
	\sum_{n=0}^{k}\Big| \sum_{j=1}^m{k  \choose n+a_j}-{k  \choose n+b_j}\Big|	\gtrsim 2^{k}k^{-\lfloor D/2 \rfloor}\geq  2^{k}k^{-c\log k}\approx 2^{k-c\log^2k}.
\end{equation*}  
Since $A$ in this case contains less than $k$ terms, contribution from $\|\chi_A'\|_p$ is at most $(2k)^{1/p}.$  Therefore it seems logical to ask whether 
\begin{equation}\label{qu}
	\|\chi_A^{(k)}\|_{l^p(\Z)}\geq C_p 2^{k-c_p\log^2k}\|\chi_A'\|_{l^p(\Z)}
\end{equation}
holds for arbitrary  $A,k.$
Our Theorem 3 and Theorem 5 obtain such growth, but only after fixing the set $A$ and letting $k\rightarrow \infty.$

We would like to highlight connections of this problem to  computer science. Observe that for a fixed $n$ we are looking at values of 
\begin{equation*}\label{}
	\begin{aligned}
	\chi_A^{(k)}(n)=\sum_{j=0}^k(-1)^{j}{k  \choose j} \chi_A(n+k-j).
	\end{aligned}
\end{equation*}
Since $\chi_A$ is either zero or one, we are indeed choosing elements from the multiset $\{(-1)^{j}{k  \choose j}:0\leq j\leq k\}$, adding them up and looking at the value. This is a variant of the  well-known subset sum problem of computational complexity theory. Variants of the subset sum problem similar to ours emerge  naturally in coding theory e.g. in the article \cite{var} of Varshamov, the problem of counting the subsets of $\{1,2,\ldots,j\}$ for which the sum of elements is $k \ (\text{mod} \ n)$ emerges. This problem has been studied in detail in \cite{wk,sy}, and exact closed formulas  were found.

Our problem looks like a discrete version of  various inequalities widely used in  PDE literature that bounds the norms of a function with the norms of its derivatives. The Sobolev imbedding theorem  and the Poincare inequality are the two most relevant results for us.  The Sobolev imbedding theorem states that the Sobolev space $W^{m,q}(\R^n)$ embeds continuously into $L^r(\R^n)$ if $p^{-1}-q^{-1}=m/n$. We note that for the case $p=q$, which is what we are interested in, this theorem gives nothing nontrivial. The Poincare inequality states that for $\Omega$ a domain bounded in at least one direction, and $1\leq p <\infty$ we have
\begin{equation*}
	\|f\|_{L^p(\Omega)}\leq C_{p,\Omega} \|\nabla f\|_{L^p(\Omega)}
\end{equation*}
for every function $f\in W^{1,p}$ with zero trace. So in this inequality $p=q$, and indeed the main idea of the inequality is that bounded support in one direction forces $f$ to change, and therefore have substantially large derivative in that direction.  This is very similar to our restriction to characteristic functions of finite sets. But as can be seen from very simple examples the constant in the Poincare inequality cannot be made independent of the domain, whereas we seek uniform estimates that will be independent of the set $A$ completely.


\section{Uniform bounds}

In this section we  present proofs of Theorem 1.  In the first subsection we  present an observation on binomial coefficients that  allows us to obtain Theorem 1 up to $k\leq 7. $ Then in subsequent subsections we will present two proofs of Theorem 1, with the second one getting a  much better bound $C_{k,p}$ than the first.

\subsection{Theorem 1 for $k\leq 7$ via an observation on binomial coefficients}

For   a finite set $A$, the   $l^p$ norms of first  two derivatives of $\chi_A$ can be computed explicitly by developing a boundary  theory for subsets of integers. This theory has been presented and utilized in \cite{tem3} as well.
We will call a subset of integers $k$-separated if any two consecutive elements of the set have between them $k$ elements that do not belong to the set. We define the left boundary   $\partial_lA$ of $A$ to be the  set of points $n$ such that $n-1\notin A$.  Similarly, $\partial_rA$ is the set of points $n\in A$ such that $n+1\notin A$. The boundary  $\partial A$ is the union of the left and the right boundaries of $A.$ 
 We observe that $\partial_lA, \partial_rA $ are $1$-separated sets, and they have the same cardinality for finite $A$.  
And  if a set $A$ is $k\geq 1$ separated, then $\partial_lA= \partial_rA=A.$ 

We observe that $\chi_A(n+1)-\chi_A(n)=1$  if and only if $n+1\in A, \ n\notin A$; that is $n \in \partial_l A-1.$ And we also observe that $\chi_A(n+1)-\chi_A(n)=-1$  if and only if $n\in A, \ n+1\notin A$; that is, $n \in \partial_r A.$  Furthermore  $ \partial_rA\cap (\partial_lA-1)=\emptyset$, for  $n\in \partial_rA$ means $n+1 \notin A.$ Thus we have
\begin{equation*}
	\chi_A'(n)=	\chi_A(n+1)-\chi_A(n)=\chi_{\partial_lA-1}(n)-\chi_{\partial_rA}(n),
\end{equation*}
and
\begin{equation*}
|\chi_A'(n)|=|\chi_{\partial_lA-1}(n)-\chi_{\partial_rA}(n)|=\chi_{\partial_rA}(n)+\chi_{\partial_lA-1}(n).
\end{equation*}
Therefore for $1\leq p <\infty$,
\begin{equation*}
\begin{aligned}
	\|\chi_A'\|_{l^p(\Z)}=\Big[\sum_{n\in \Z}(\chi_{\partial_rA}(n)+\chi_{\partial_lA-1}(n))^p\Big]^{1/p}=\Big[\sum_{n\in \Z}\chi_{\partial_rA}(n)+\chi_{\partial_lA-1}(n)\Big]^{1/p}=(2|\partial_rA|)^{1/p}.
		\end{aligned}
\end{equation*}
Plainly this equality also holds for $p=\infty$.

We can also calculate the second derivative using these properties.
\begin{equation*}
	\begin{aligned}
	\chi_A''(n)=	\chi_A'(n+1)-\chi_A'(n)&=\chi_{\partial_lA-2}(n)-\chi_{\partial_rA-1}(n)-[\chi_{\partial_lA-1}(n)-\chi_{\partial_rA}(n)]\\ &= \chi_{\partial_lA-2}(n)+\chi_{\partial_rA}(n)-\chi_{\partial_rA-1}(n)-\chi_{\partial_lA-1}(n)
	\end{aligned}
\end{equation*}
We claim that $(\partial_lA-2)\cup \partial_rA$ does not intersect  $(\partial_rA-1)\cup (\partial_lA-1)$.  To see this, let $n\in \partial_rA.$ We know that it cannot be in $\partial_lA-1$. It cannot belong to   $ \partial_rA-1$ as well, for $\partial_rA$ is $1$ separated. For the same reasons $n\in \partial_lA-2$ cannot be in  
$(\partial_rA-1)\cup (\partial_lA-1)$, yielding the claim. Therefore
\begin{equation*}
	\begin{aligned}
		|\chi_A''(n)|&=	 |\chi_{\partial_lA-2}(n)+\chi_{\partial_rA}(n)-\chi_{\partial_rA-1}(n)-\chi_{\partial_lA-1}(n)| \\ &=\chi_{\partial_lA-2}(n)+\chi_{\partial_rA}(n)+\chi_{\partial_rA-1}(n)+\chi_{\partial_lA-1}(n).
	\end{aligned}
\end{equation*}
Therefore for $1\leq p<\infty,$
\begin{equation*}
	\begin{aligned}
		\|\chi_A''\|_{l^p(\Z)} &= \Big[\sum_{n\in \Z}\big(\chi_{\partial_lA-2}(n)+\chi_{\partial_rA}(n)\big)^p+\big(\chi_{\partial_rA-1}(n)+\chi_{\partial_lA-1}(n)\big)^{p}  \Big]^{1/p}
		\\ &\geq  \Big[\sum_{n\in \Z}\chi_{\partial_lA-2}(n)+\chi_{\partial_rA}(n)+\chi_{\partial_rA-1}(n)+\chi_{\partial_lA-1}(n) \Big]^{1/p}\\ &= (4|\partial_rA|)^{1/p}.
	\end{aligned}
\end{equation*}
But also 
\begin{equation}\label{2nd}
	\|\chi_A''\|_{l^p(\Z)}\leq 2\|	\chi_A'\|_{l^p(\Z)}=2^{1+{1}/{p}}|\partial_rA|^{1/p}.
\end{equation}
Again we observe that these bounds also hold for $p=\infty.$ We also observe that both these bounds are sharp. For  the set $A=\{0,1\}$  we have  $	\|\chi_A''\|_{l^p(\Z)}=(4|\partial_rA|)^{1/p}$  for any $1\leq p\leq \infty$. For the sets $A_n=\{2,4,6,\ldots,2n\}$ we have $|\partial_rA|=n$, and  $	\|\chi_A''\|_{l^p(\Z)}=[2^p(2|\partial_rA|-1)+2]^{1/p}$, showing that the constant $2^{1+1/p}$ cannot be replaced by any smaller constant.

   As higher derivatives concern farther apart points, this nonintersection argument begins to lose its force.  But it still  allows us to obtain lower bounds for $l^p$ norms of derivatives for up to a certain order.

For  the third derivative
 \begin{equation}\label{lb3}
 \chi_A^{(3)}(n):=\chi_A(n+3)-3\chi_A(n+2)+3\chi_A(n+1)-\chi_A(n),
 \end{equation}
if we consider $n\in \partial_rA -1$, then $n+1\in A$ but $n+2\notin  A$.  For these $n$, therefore,   \eqref{lb3}
 is not less than $2$. Similarly if $n\in \partial_lA -2$, then $n+2\in A$ but $n+1\notin A$.  For these $n$, then,  \eqref{lb3} is not more than -2. As $\partial_rA\cap (\partial_lA -1)=\emptyset$,
 \begin{equation}\label{}
 \|	\chi_A^{(3)}(n)\|_{l^p(\Z)}^p\geq \sum_{n\in  (\partial_rA -1)\cup(\partial_lA -2) }|\chi_A(n+3)-3\chi_A(n+2)+3\chi_A(n+1)-\chi_A(n)|^p=2^{1+p}|\partial_rA|.
 \end{equation}
Thus,  for any $1\leq p\leq \infty$ we obtain
$
	\|\chi_A^{(3)}\|_{l^p(\Z)}\geq 2^{1+{1}/{p}}|\partial_rA|^{1/p}.
$

For $k=4,$ the same argument  yields
$|\chi_A^{(4)}(n)|\geq 2 $
if we pick $n\in \partial_rA-1,\partial_lA-2 .$
Then we obtain
$
	\|\chi_A^{(4)}\|_{l^p(\Z)}\geq 2^{1+{1}/{p}}|\partial_rA|^{1/p}.
$

For  $k=5$, picking  $n\in \partial_rA-2,\partial_lA-3$  yields
$
	|	\chi_A^{(5)}(n)|\geq 4,
$
implying
$
	\|\chi_A^{(5)}\|_{p}\geq 2^{2+{1}/{p}}|\partial_rA|^{1/p}.
$

For  $k=6$, we have
$	|	\chi_A^{(6)}(n)|\geq 3$
 if we pick $n\in \partial_rA-2,\partial_lA-3.$ 
Thus
$
	\|\chi_A^{(6)}\|_{p}\geq 3(2|\partial_rA|)^{1/p}.
$

For  $k=7$, we have
$
		|\chi_A^{(7)}(n)|\geq 6
		$
if we pick $n\in \partial_rA-3,\partial_lA-4.$ 
Hence
$
	\|\chi_A^{(7)}\|_{p}\geq 6(2|\partial_rA|)^{1/p}.
$

The method is stuck at 8th derivative. Clearly we need further arguments to proceed.

\subsection{Theorem 1 in  $p=\infty$ case} In this subsection we prove the following special case of Theorem 1.  The proof is  similar to the two proofs of  Theorem 1 that we will present, but is simpler. The trade-off of simplicity is that it works only for $l^{\infty}$.

\begin{proposition}
	Let $A$ be a finite nonempty subset of integers. For any $k\geq 1$,
	\begin{equation}
	\|\chi_A^{(k)}\|_{l^\infty(\Z)}\geq 1=\|\chi_A'\|_{l^\infty(\Z)}.
\end{equation} 	
	\end{proposition}

\begin{proof}
	We observe that if $\chi_A^{(k)}=0$ identically, then $\chi_A^{(k-1)}$ is constant, and this constant must be zero,   for as $A$ is finite, the support of  $\chi_A^{(k-1)}$  is also finite, hence it cannot be a constant other than zero. Iterating this argument yields a contradiction as we assumed $A$ to be a finite nonempty set. Hence $\chi_A^{(k)}$ cannot be identically zero. As it is integer valued, the claim follows. 
\end{proof}

 \subsection{The First proof of Theorem 1} 
 We start with  some simple observations. Clearly if $n\in  \partial_rA$, then  $\chi_A (n),\chi'_A (n),\chi''_A (n)\neq 0.$ 
 Also if $n\in  \partial_lA-1$, then clearly $\chi'_A (n),\chi''_A (n)\neq 0.$ 
 
 \begin{proposition}
 	Let $k\geq 3$.	Consider an interval $I=\{ n,n+1,\ldots,n+2^k-1 \}$ that contains a point from  $\partial_rA\cup (\partial_lA-1)$. Then $\chi^{(k)}_A$ cannot be zero on all of this interval.
 \end{proposition}
 
 \begin{proof}
 	Assume to the contrary that 	 $\chi^{(k)}_A$ is zero on all of $I$.  Then $\chi^{(k-1)}_A$ is constant on all of $I\cup\{n+2^k\}.$ Now if  this constant is nonzero,
 	then as $\chi_A^{(k-1)}(n+j+1)-\chi_A^{(k-1)}(n+j)=\chi_A^{(k)}(n+j)=0$ we have $\chi_A^{(k-1)}(n+j)$  a constant $K$ for any $0\leq j\leq 2^k,$ and this constant is an integer. Therefore it is at least 1 in absolute value. Then, as 
 	$\chi_A^{(k-2)}(n+j+1)-\chi_A^{(k-2)}(n+j)=\chi_A^{(k-1)}(n+j)=K,$ we have $\chi_A^{(k-2)}(n+j)=jK+\chi_A^{(k-2)}(n)$ for any $0\leq j\leq 2^k+1$. So $|\chi_A^{(k-2)}(n+2^k+1)-\chi_A^{(k-2)}(n)|\geq 2^{k+1}.$ Therefore at least one of $|\chi_A^{(k-2)}(n+2^k+1)|,|\chi_A^{(k-2)}(n)|$ is at least $2^k.$ But as we know  $\|\chi^{(k-2)}_A\|_{l^\infty(\Z)}\leq 2^{k-2}$, this is a contradiction.
 	
 	So this constant is zero. Then we iterate the argument above  until we obtain that  the second derivative is zero on $I$, which is a contradiction as $I$ contains a point from   $\partial_rA\cup (\partial_lA-1)$.

 \end{proof}
 
 We now employ this proposition to prove Theorem 1.
 
 \begin{proof}[The first proof of Theorem 1]
 Let us partition 
 \begin{equation*}
 \Z=\bigcup_{n\in \Z} I_n,   \qquad  \qquad    I_n=\big\{j\in \Z:  n2^k\leq j <(n+1)2^k \big\}.
 \end{equation*}
 Out of these $I_n$,
 at least $\lceil | \partial_rA|/2^{k-1}\rceil $  intervals contain elements from $\partial_rA\cup (\partial_lA-1)$, for  otherwise, as each $I_n$ can contain at most $2^k$ elements, the total number of $2|\partial_rA|$ could not be found. On these intervals  the function $\chi_A^{(k)}$ is not identically zero. Therefore $\|\chi_A^{(k)}\|_{l^\infty(\Z)}\geq 1$, and for $1\leq p<\infty$ the following concludes the proof
 \begin{equation*}
 	\|\chi_A^{(k)}\|_{l^p(\Z)}^p=\sum_{n\in \Z}\sum_{j\in I_n}|\chi_A^{(k)}|^p\geq 2^{1-k}|\partial_rA|=2^{-k}\|\chi_A'\|_{l^p(\Z)}^p.
 \end{equation*}
 \end{proof}
 
 Plainly this is a very poor bound, but it is uniform as the constant $2^{-k/p}$ is independent of $A$.
 
 \subsection{The second proof of Theorem 1}
 We start with a proposition similar in spirit to the one in the previous subsection. For $h\in \R, \  j\in \Z^+$ we define  descending factorial powers  $h^{\underline{j}}:=h(h-1)(h-2)\cdots(h-j+1)$ if $j>0$, and $h^{\underline{j}}=1$ if $j=0$.

 \begin{proposition}
 	Let $k\geq 3$.	Let $I=\{ n+k-1,n+k,\ldots,n+3k-1 \}$ be an interval that contains a point from  $\partial_rA\cup (\partial_lA-1)$. Then $\chi^{(k)}_A$ cannot be zero on all of $I'=I-(k-1)=\{n,n+1,\ldots,n+2k\}$.
 \end{proposition}
 
 \begin{proof}
 	Suppose $\chi^{(k)}_A$ is zero on all of $\{n,n+1,\ldots,n+2k\}$. Then $\chi^{(m)}_A(n)$ is zero for any $k\leq m\leq 3k.$ Then  for $h\in \{k-1,k,k+1,\ldots, 3k\}$ we have by the  discrete Taylor expansion
 	\begin{equation*}
 		\chi_A(n+h)=\sum_{j=0}^{h}\frac{	\chi_A^{(j)}(n)}{j!}h^{\underline{j}}=\sum_{j=0}^{k-1}\frac{	\chi_A^{(j)}(n)}{j!}h^{\underline{j}}.
 	\end{equation*}
 	So if we set for all real $h$
 	\begin{equation*}
 		P(h)=\sum_{j=0}^{k-1}\frac{	\chi_A^{(j)}(n)}{j!}h^{\underline{j}},
 	\end{equation*}
 	this is a polynomial in the real variable $h$ of degree at most   $k-1$, and it takes the values $0,1$ a total of $2k+2$ times when $h\in \{k-1,k,k+1,\ldots, 3k\} $, which is impossible by the fundamental theorem of algebra, unless it is constant. But as $I$ contains a point from $\partial_rA\cup (\partial_lA-1)$, the polynomial   $P(h)$ takes both values $0,1$ at least once. This contradiction concludes the proof.
 	
 \end{proof}
 
 Employing this proposition we prove Theorem 1 with a better bound.
 
 \begin{proof}
 	Let us partition 
 	\begin{equation*}
 		\Z=\bigcup_{n}I_n, \qquad   \qquad  I_n:=\big\{j\in \Z:  n(2k+1)\leq j <(n+1)(2k+1)\big\}.
 	\end{equation*}
 	Out of these $I_n$,
 	at least $\lceil 2|\partial_rA|/(2k+1)\rceil $  intervals contain elements from $\partial_rA\cup (\partial_lA-1)$.   For each $I_n$ with a $\partial_rA\cup (\partial_lA-1)$ point,   $\chi_A^{(k)}$ is not identically zero on $I_n-(k-1)$. This concludes the theorem for $p=\infty$. For $1\leq p<\infty$ the following concludes the proof
 	\begin{equation*}
 		\|\chi_A^{(k)}\|_{l^p(\Z)}^p=\sum_{n\in \Z}\sum_{j\in I_n-(k-1)}|\chi_A^{(k)}(j)|^p\geq 2|\partial_rA|/(2k+1)\geq (2k+1)^{-1}\|\chi_A'\|_{l^p(\Z)}^p.
 	\end{equation*}

 \end{proof} 
 This is a uniform bound that improves greatly upon the previous uniform bound. But it is still far away from an exponential uniform bound that we suspect should be possible.

 \subsection{Proof of Theorem 2}
 
 Although Theorem 2 gives a very nice exponential bound, because it relies only on endpoints of $A$, the bound is not uniform.
 
\begin{proof} Let $a_l$ be the largest element of $A$.  For any natural number $n\leq k/3$ 
 \begin{equation*}
 	{k\choose n}={k \choose n-1}\frac{k-n+1}{n}>2{k \choose n-1}.
 \end{equation*}
 Therefore 
 \begin{equation*}
 	{k\choose n}-\sum_{j=0}^{\lfloor \frac{n-1}{2} \rfloor}{k\choose n-2j-1}\geq {k\choose n}\big[1 -2^{-1}-2^{-3}-2^{-5}\cdots  \big]=\frac{1}{3}{k\choose n}.
 \end{equation*}
 So we can write
 \begin{equation*}
 	\begin{aligned}
 		|\chi_A^{(k)}(a_l-n)|=\Big| \sum_{j=0}^{k}{k\choose j}(-1)^j\chi_A(a_l-n+k-j)   \Big|=\Big| \sum_{j=k-n}^{k}{k\choose k-j}(-1)^j\chi_A(a_l-n+k-j)   \Big| .
 	\end{aligned}
 \end{equation*}
 After a change of variable we obtain the following that concludes the proof
 \begin{equation*}
 	\begin{aligned}
 		\geq \Big| \sum_{j=0}^{n}{k\choose j}(-1)^{k-j}\chi_A(a_l-n+j)   \Big|  
 		\geq \Big| {k\choose n}-\sum_{j=0}^{\lfloor \frac{n-1}{2} \rfloor}{k\choose n-2j-1}  \Big|  
 		\geq \frac{1}{3}{k\choose n}.
 	\end{aligned}
 \end{equation*}

 \end{proof}

 
\section{The Fourier approach}

We now present the Fourier transform approach that  allows us to prove nonuniform results with growth  as in \eqref{qu}.
 For any complex-valued finitely supported function $f:\Z \rightarrow \C$  the Fourier series 
\[
\widehat{f}(x) = \sum_{n \in \mathbb{Z}} f(n) e^{-2\pi i n x}, \quad x  \in \mathbb{T},
\]
 is a finite sum. Hence, if we consider derivatives
\[
\widehat{f'}(x) = \sum_{n \in \mathbb{Z}} [f(n+1)-f(n)] e^{-2\pi i n x}=[e^{2\pi i  x}-1]\sum_{n \in \mathbb{Z}} f(n) e^{-2\pi i n x}=[e^{2\pi i  x}-1] \widehat{f}(x).
\]
Iterating this
\[
\widehat{f^{(k)}}(x) =[e^{2\pi i  x}-1]^k \widehat{f}(x).
\]
We  observe that 
\begin{equation*}
	\begin{aligned}
		|e^{2\pi ix}-1|=\big[ (\cos 2\pi x -1)^2 +\sin^2 2\pi  x \big]^{\frac{1}{2}}=\big[2-2\cos 2\pi x  \big]^{\frac{1}{2}}= 2\sin \pi x.
	\end{aligned}
\end{equation*}
We have 	$2\sin \pi x   >1 $  if $  1/6<x<5/6$,     and    	$2\sin \pi x  \leq 1$
otherwise. Then for $p\in [1,2]$ using the Hausdorff-Young inequality  
\begin{equation}\label{dv}
	\begin{aligned}
		\|\chi_A^{(k)}\|_{l^{p}(\Z)} \geq  \|\widehat{\chi_A^{(k)}}\|_{L^{p'}(\Tt)}   & =  \Big\||e^{2\pi ix}-1|^k     \sum_{n \in A}  e^{-2\pi i n x}\Big\|_{L^{p'}(\Tt)}  \\ & =  \big\| \big[2\sin \pi x \big]^k   \widehat{\chi_A}\big\|_{L^{p'}(\Tt)} \\ &\geq   \big\| \big[2\sin \pi x \big]^k   \widehat{\chi_A}\big\|_{L^{p'}((\frac{1}{2}-r,\frac{1}{2}+r))},
	\end{aligned}
\end{equation}
for some small $r>0.$  Taking this as a starting point we will speculate about the dynamics of $\|\chi_A^{(k)}\|_{l^p(\Z)}$ as $k\rightarrow \infty.$ 
 Multiplication by $2\sin \pi x $  boosts the  exponential polynomial around $1/2$, while around $0$   suppresses it. But any exponential polynomial has around zero a small interval where the size  is like $|A|$, while on the rest of $\Tt$ it is much smaller, and on average it is like $|A|^{1/2}$.  Therefore the norm  of the product depends very much on whether the norm of the exponential polynomial concentrates.  

If a set $A$ is sparse, then the orthogonality behaviour of the characters gets even stronger, and they behave as if  independent random variables, sum of which with large probability concentrates around the mean.  So, for such sets $L^p(\Tt)$ norm is dominated by  the rest of $\Tt$ up to some $1<q\leq \infty$, but then a small neighborhood of the origin dominates. So, for such a set   multiplying by $2\sin \pi x $ boosts the norms. 

But for a tightly concentrated set $A$, such as an interval, the concentration around $0$ can be very strong, and the function can be like $|A|$ for a neighborhood of size $\approx |A|^{-1}$, and like $1$ for the rest of $\Tt.$ In that case the neighorhood of $0$ dominates for any  $p\in (1,\infty]$. For such a function multiplying with $|e^{2\pi i x}-1|$ erases the contribution from the neighborhood of  $0$, as $|e^{2\pi i x}-1|\approx x$ for small $x$, which means on  a $|A|^{-1}$ size neighborhood of $0$ it is like $|A|^{-1},$ reducing the values of exponential polynomial from $|A|$ to $1$ at once. 
Contribution from the rest of the torus is boosted. But as this boosting is by at most 2, while the dominating piece is essentially  demolished, the norm significantly decreases after multiplication by $|e^{2\pi i x}-1|$ once.  After this point, if we keep multiplying by $|e^{2\pi i x}-1|$, since contribution from the neighborhood of zero is already gone, there is no more loss from this part. So norm can be expected to grow after this point. All this is very clear if we take $A$ to be an interval. 

In short, the Fourier approach demonstrates that,  for an arbitrary  set $A,$ while $\|\chi_A'\|_{l^p(\Z)}$ cannot be expected to dominate $\|\chi_A\|_{l^p(\Z)}$, it may be expected that $\|\chi_A^{(k)}\|_{l^p(\Z)}/\|\chi_A'\|_{l^p(\Z)}$ will grow    exponentially.

We now continue from \eqref{dv}. 
We observe that  on $(\frac{1}{2}-r,\frac{1}{2}+r)$ we have 
\begin{equation*}
	|e^{2\pi i  x}-1| \geq |\cos 2\pi (\frac{1}{2}-r) -1|\geq |\cos (\pi-2\pi r) -1|.
\end{equation*}
By Taylor's theorem
\begin{equation*}
	\cos (\pi-2\pi r)= -1+2\pi^2r^2-\frac{4}{3}\pi^3\sin(\pi-2\pi s) r^3, \qquad  \text{for \ some }\quad   0\leq s\leq r.
\end{equation*}
Plugging this in
\begin{equation*}
	|-2+2\pi^2r^2-\frac{4}{3}\pi^3\sin(\pi-2\pi s) r^3|\geq 2-2\pi^2r^2.
\end{equation*}
Hence  the last term of  \eqref{dv} satisfies 
\begin{equation}\label{dv2}
\geq 2^k(1-\pi^2r^2)^k  \Big\|   \sum_{n \in A}  e^{-2\pi i n x}\Big\|_{L^{p'}((\frac{1}{2}-r,\frac{1}{2}+r))}.
\end{equation}
To continue after this point we have three methods: the Nazarov-Tur\'{a}n bound, the Borwein-Erd\'{e}lyi bound, and the Borwein-Erd\'{e}lyi-K\'{o}s observation.  The first two of these are valid for any interval, while the last is specific to intervals around $1/2$. We will apply each of these in its own subsection below.

\subsection{The Nazarov-Tur\'{a}n bound} This is a  Remez-type inequality that Tur\'{a}n \cite{pt} proved for intervals, generalized by Nazarov \cite{naz1} to arbitary measurable subsets.  Let  $A\subset \Z$
be a finite nonempty set, and $E\subseteq \Tt$ be a measurable set. Then the Nazarov-Tur\'{a}n bound states that for complex coefficients $a_n$, 
\begin{equation}
	\Big\|  \sum_{n\in A}a_ne^{2\pi i n x}	\Big\|_{L^{\infty}(\Tt)}\leq \Big(\frac{14e}{|E|} \Big)^{|A|-1}	\Big\|  \sum_{n\in A}a_ne^{2\pi i n x}	\Big\|_{L^{\infty}(E)}.
	\end{equation}
With this at hand we are ready to  prove Theorem 3.

\begin{proof}[Proof of Theorem 3]

If $|A|=1$, then  \eqref{nb2} follows from \eqref{nb1}, and \eqref{nb1} follows from the trivial $p=1$ case and  \eqref{hol}. So we assume $|A|>1.$
With the Nazarov-Tur\'{a}n bound \eqref{dv2} satisfies
\begin{equation}\label{nt}
	\begin{aligned}
	\|\chi_A^{(k)}\|_{l^{1}(\Z)}   \geq 2^{k}(7e)^{1-|A|}(1-\pi^2r^2)^kr^{|A|-1}  	|A|.
	\end{aligned}
\end{equation}
We want to choose $r$ so as to maximize $(1-\pi^2r^2)^k r^{|A|-1}$. For notational simplicity, and for future use, we will maximize $(1-\pi^2r^2)^k r^{\alpha}$ for any $\alpha>0$, and $r\in(0,\pi^{-1}).$
This function is zero at $r=0,\pi^{-1}$
and attains its maximum on $[0,\pi]$  at
\begin{equation}\label{alp-}
r=\frac{1}{\pi}\Big[ \frac{\alpha}{2k+\alpha}  \Big]^{1/2},
\end{equation}
and this maximum is given by 
\begin{equation}\label{alp}
\frac{1}{\pi^{\alpha}}\Big[\frac{2k}{2k+\alpha}   \Big]^k  \Big[\frac{\alpha}{2k+\alpha}   \Big]^{\frac{\alpha}{2}}.
\end{equation}
So the  right hand side of \eqref{nt} satisfies
\begin{equation}\label{nt2}
	\begin{aligned}
	\geq 2^k\Big(\frac{1}{7\pi e}  \Big)^{|A|-1}\Big[\frac{2k}{2k+|A|-1}   \Big]^k  \Big[\frac{|A|-1}{2k+|A|-1}   \Big]^{\frac{|A|-1}{2}}|A|.
	\end{aligned}
\end{equation}

For a fixed $\alpha$ we can compute the limits of two factors in \eqref{alp} as $k\rightarrow \infty$.
We observe by using the Taylor expansion at 0 for $\log (1+x),$ 
\begin{equation*}\label{lim1}
	\begin{aligned}
	\lim_{k\rightarrow \infty} \log \Big[\frac{2k}{2k+\alpha}   \Big]^k  =	\lim_{k\rightarrow \infty}k  \Big[-\frac{\alpha}{2k+\alpha}-  \frac{1}{2}\Big(\frac{\alpha}{2k+\alpha}\Big)^2+\ldots \Big] =  -\frac{\alpha}{2}.
	\end{aligned}
\end{equation*}
  So we have
\begin{equation*}\label{lim2}
	\begin{aligned}
		\lim_{k\rightarrow \infty}  \Big[\frac{2k}{2k+\alpha}   \Big]^k =e^{-\frac{\alpha}{2}}.
	\end{aligned}
\end{equation*}
On the other hand  
\begin{equation*}\label{lim3}
	\begin{aligned}
	\Big[\frac{\alpha}{2k+\alpha}   \Big]^{\frac{\alpha}{2}}= \Big[ \frac{1}{2k}\Big]^{\frac{\alpha}{2}}	\Big[\frac{\alpha}{1+\frac{\alpha}{2k}}   \Big]^{\frac{\alpha}{2}}=2^{-\frac{\alpha}{2}\log_2 2k}(\alpha)^{\frac{\alpha}{2}}\Big[\frac{2k}{2k+\alpha}   \Big]^{\frac{\alpha}{2}}.
	\end{aligned}
\end{equation*}
The last factor goes to 1 as $k\rightarrow \infty$. Thus   \eqref{nt2} satisfies as $k\rightarrow \infty$
\begin{equation}\label{lim4}
	\begin{aligned}
\sim 	2^{k-\frac{|A|-1}{2}\log_2 2k}\Big(\frac{\sqrt{|A|-1}}{7\pi e^{3/2} }  \Big)^{|A|-1}|A|.
	\end{aligned}
\end{equation}
This proves \eqref{nb1} for $p=1$. For higher $p$ it follows from \eqref{hol}.

The Nazarov-Tur\'{a}n bound can also be used via the boundary of $A$ as follows. 

\begin{equation*}\label{nt3}
	\begin{aligned}
		\big\| [e^{2\pi i  x}-1]^{k+1}  \widehat{\chi_A}\big \|_{L^{\infty}((\frac{1}{2}-r,\frac{1}{2}+r))}&= \big\| [e^{2\pi i  x}-1]^{k} \Big[ \sum_{n \in \partial_{l}A-1}e^{-2\pi i n x} -\sum_{n \in \partial_rA} e^{-2\pi i n x}\Big]\big \|_{L^{\infty}((\frac{1}{2}-r,\frac{1}{2}+r))}   \\&\geq (2-2\pi^2r^2)^{k}\Big\|   \sum_{n \in \partial_{l}A-1}e^{-2\pi i n x} -\sum_{n \in \partial_rA} e^{-2\pi i n x}\Big\|_{L^{\infty}((\frac{1}{2}-r,\frac{1}{2}+r))}\\ &\geq 2^{k}(7e)^{1-2|\partial_lA|}(1-\pi^2r^2)^{k}r^{2|\partial_{l}A|-1}  	2|\partial_{l}A|.
	\end{aligned}
\end{equation*}
The  same steps then will then  yield \eqref{nb2}.

\end{proof}

It does not seem possible to reach a uniform estimate from this method, as picking a small $k$ like $k=10$ and picking $|A|$ large we see plainly that the constant in \eqref{nt2}  is extremely small  and determined essentially by the term  $(7\pi e)^{1-|A|}$, which  is impossible to bound  below  with a function depending on  $k$.

\subsection{ The  Borwein-Erd\'{e}lyi  bound}

There is another Remez-type inequality due to Borwein and Erd\'{e}lyi \cite{be}  that can be applied in our situation. Let 

\begin{equation*}
\mathcal{F}:=\Big\{ \sum_{j=0}^na_jz^j:   a_j\in \{ -1,0,1  \}  \Big\}.
\end{equation*}
Then the Borwein-Erd\'{e}lyi result gives for an arc $E\subset \Tt$, and an absolute constant $c$
\begin{equation*}
	\inf_{\substack{P\in \mathcal{F}\setminus \{0\}}}\|P\|_{L^{\infty}(E)}\geq e^{-c/|E|}.
\end{equation*}
Here there is no dependency on the degree of the polynomial. We may assume  $c\geq 1$  as otherwise it can be replaced by 1. 

\begin{proof}[Proof of Theorem 4]

We replace the Borwein-Erd\'{e}lyi bound into \eqref{nt} to obtain
\begin{equation}\label{be}
	\begin{aligned}
		\big\| {\chi_A^{(k)}}\big \|_{l^{1}(\Z)}  \geq 2^k(1-\pi^2r^2)^k e^{-\frac{c}{2r}}.
	\end{aligned}
\end{equation}
 We want  to maximize $(1-\pi^2r^2)^ke^{-\frac{c}{2r}}$. This function vanishes at $r=\pi^{-1}$ and as $r\downarrow 0$,
and differentiating it we get
\begin{equation*}
	-(1-\pi^2r^2)^{k-1}e^{-\frac{c}{2r}}(2r)^{-2}\big[4k\pi^2r^3+  c\pi^2r^2-c  \big]=-(1-\pi^2r^2)^{k-1}e^{-\frac{c}{2r}}(2r)^{-2}\phi(r).
\end{equation*}
As the factors except $\phi(r)$ are plainly nonzero on $(0,\pi^{-1})$, we just need to seek zeros of $\phi$ on this interval. As it is a strictly increasing here with $\phi(0)=-c, \ \phi(\pi^{-1})=4k/\pi,$ it has a single zero. 
We will give a pretty good approximation for the root, though we will not solve for it exactly. 
 We observe that  if  $4kr < Cc$ for some large constant $C$, then for $k$ large enough
\begin{equation*}
\phi(r)	<c\Big[\pi^2(C+1)\frac{C^2c^2}{16k^2}-1\Big]<0.
\end{equation*}
 So the  zero of the function cannot lie in the interval 
$\big (0,{Cc}(4k)^{-1}\big).$
Outside this interval we have 
\begin{equation*}
 4k\pi^2r^3	\leq 4k\pi^2r^3 + c\pi^2r^2  \leq (1+C^{-1}) 4k\pi^2r^3.
\end{equation*} 
So the zero of our function is  
\begin{equation*}
(1-\varepsilon)\Big(\frac{c}{4k\pi^2}  \Big)^{1/3}.
\end{equation*}
for some small $\varepsilon>0$. As an approximation, we just take $\varepsilon=0$, and plug   this  in  to get 
\begin{equation*}
\sup_{r\in [0,\pi^{-1}]}	2^k(1-\pi^2r^2)^k e^{-\frac{c}{2r}}\geq 2^k\Big[1-\pi^2\Big(\frac{c}{4k\pi^2}  \Big)^{\frac{2}{3}}\Big]^k e^{-c^{\frac{2}{3}}(k\pi^2)^{\frac{1}{3}}2^{-\frac{1}{3}}}.
\end{equation*}
We have
\begin{equation*}
 \begin{aligned}
 \log\Big[1-\pi^2\Big(\frac{c}{4k\pi^2}  \Big)^{\frac{2}{3}}\Big]^k =k \log\Big[1-\pi^2\Big(\frac{c}{4k\pi^2}  \Big)^{\frac{2}{3}}\Big]&=k
 \Big[ -\pi^2\Big(\frac{c}{4k\pi^2}\Big)^{\frac{2}{3}}+ \mathcal{O}(k^{-4/3}) \Big]\\ &= 
 -\Big(\frac{c\pi}{4}\Big)^{\frac{2}{3}}k^{\frac{1}{3}}+ \mathcal{O}(k^{-1/3}). 
 \end{aligned}
\end{equation*}
This means 
\begin{equation*}
	\begin{aligned}
		\Big(1-\pi^2\Big(\frac{c}{4k\pi^2}  \Big)^{\frac{2}{3}}\Big)^k \sim 2^
		{-(\log_2e)\big(\frac{c\pi}{4}\big)^{\frac{2}{3}}k^{\frac{1}{3}}}. 
	\end{aligned}
\end{equation*}
Plugging this  in we  get the following asypmtotics:
\begin{equation*}
 2^k\Big[1-\pi^2\Big(\frac{c}{4k\pi^2}  \Big)^{\frac{2}{3}}\Big]^k e^{-c^{\frac{2}{3}}(k\pi^2)^{\frac{1}{3}}2^{-\frac{1}{3}}}\sim 2^{k-3(\log_2e)\big(\frac{c\pi}{4}\big)^{\frac{2}{3}}k^{\frac{1}{3}}}.
\end{equation*}
This together with \eqref{hol} finishes \eqref{be1}.

\end{proof}

\subsection{Borwein-Erd\'{e}lyi-K\'{o}s observation}

As we have seen,  it is of great importance to estimate  $|\widehat{\chi_A}(x)|$ around
  $x=1/2$. Indeed, we will look at the  simpler exponential sum
\begin{equation*}
	P_A(x):=  \sum_{n \in A_*}  e^{2\pi i n x},   
\end{equation*}
where $A_*$ is obtained by shifting the leftmost element  of $A$ to $0.$ Then
\begin{equation*}
	|P_A(x)|= \Big| \sum_{n \in A_*}  e^{2\pi i n x} \Big|  =\Big| \sum_{n \in A}  e^{2\pi i n x} \Big| = |\widehat{\chi_A}(-x)|,
\end{equation*}
and by periodicity 
\begin{equation*}
	\big|\widehat{\chi_A}\big(1/2+r\big)\big|=\big|P_A(1/2-r)\big|.
\end{equation*}
As the function $P_A$ is real analytic, a very good way of  estimating  its size in a small neighborhood of $1/2$ is the Taylor expansion around this point.  So 
\begin{equation*}
	P_A(x)= \sum_{j=0}^{\infty}P_A^{(j)}(1/2)\frac{(x-1/2)^{j}}{j!}. 
\end{equation*}
Since for $x$ close enough to $1/2$ the lowest order nonzero term of this series will dominate, we would like to understand  up to  how many orders $P_A^{(j)}(1/2)$ may vanish.  We have
\begin{equation*}
	P_A^{(j)}(x)	= (2\pi i )^j\sum_{n\in A_*}n^j e^{2\pi inx},  \qquad  P_A^{(j)}(1/2)	= (2\pi i)^j\sum_{n\in A_*}(-1)^nn^j.
\end{equation*}
For this to vanish up to order $J$ we must have
\begin{equation}\label{pte1}
	\sum_{n\in A_*}(-1)^nn^j=0 \qquad    0\leq j\leq J.
\end{equation}
As mentioned in the introduction, this is a special version of the  Prouhet-Tarry-Escott problem.

We may also view $P_A$ as restriction of a complex polynomial to the torus:

\begin{equation*}
	P_A(x)= \sum_{n \in A_*}  e^{2\pi i n x}=\sum_{n \in A_*}  z^{n}\Big|_{\Tt}=P_A(z)\Big|_{\Tt}.
\end{equation*}
But we  expand this polynomial around $-1$ to obtain
\begin{equation*}
	P_A(z)=\sum_{n \in A_*}  z^{n}=  \sum_{n=0}^{\infty} P_A^{(j)}(-1)\frac{(z+1)^j}{j!} .
\end{equation*}
So  we should look at derivatives of this complex series at $-1.$ 
\begin{equation*}
	P_A^{(j)}(-1)=\sum_{n \in A_*}  (-1)^{n-j}n^{\underline{j}}.
\end{equation*}
 For this to vanish up to order $J$ we must have
\begin{equation}\label{pte2}
	\sum_{n \in A_*} (-1)^{n}n^{\underline{j}} =0  \qquad    0\leq j\leq J.
\end{equation}
This holds if and only if \eqref{pte1} holds. Also it holds if and only if 
\begin{equation}\label{pte3}
	P_A(z)=(z+1)^{J+1}G(z),
\end{equation}
where $G$ is an integer coefficient complex polynomial. From this a good bound on $J$ can be deduced by a short elegant argument  of Borwein, Erd\'{e}lyi, K\'{o}s  \cite{bek}. Plugging in $1$ to both sides of \eqref{pte3},
\begin{equation}\label{pte4}
	|A|=P_A(1)=2^{J+1}G(1).
\end{equation}
As $G$ has integer coefficients we have $|G(1)|\geq 1.$ So 
\begin{equation*}\label{}
	|A|=2^{J+1}|G(1)| \geq 2^{J+1},
\end{equation*}
implying that $J\leq -1+\log_2|A|.$

This gives for fixed $A$ an asymptotics for $\|\chi_A^{(k)}\|_p$ that is better than what we obtained from the Nazarov-Tur\'{a}n or the Borwein-Erd\'{e}lyi methods. For those bounds are devised for arbitrary arcs of $\Tt$, while this method is optimal for an arc around $1/2.$ As there is  a number $a\leq \log_2|A|$ for which $P_A^{(a)}(1/2)\neq 0$,  we have
for $x=1/2+r$
\begin{equation}\label{pte5}
|\widehat{\chi_A}(1/2+r)|=	|P_A(1/2-r)|=\Big|P^{(a)}_A(1/2)\frac{(-r)^a}{a!}+P^{(a+1)}_A(1/2-\rho)\frac{(-r)^{a+1}}{a+1!}\Big|,
\end{equation} 
where $0\leq |\rho|\leq |r|$. Now by \eqref{pte1} the coefficient $|P^{(a)}_A(1/2)|$ can be as small as $(2\pi)^a$, while $|P^{(a+1)}_A(1/2-\rho)|$ can be as large as $(2\pi)^{a+1}|A|(\max A_*)^{a+1}$. So although the first term in the right hand side of \eqref{pte5} dominates as $r\downarrow0$, this happens after a point that depends on the set $A.$  Let $r_A>0$ be such that for $|r|\leq r_A$ the second term is at most   half of the first term. For such $r$,
\begin{equation*}\label{}
	|\widehat{\chi_A}(1/2+r)|=	|P_A(1/2-r)|\geq 2^{a-1}\pi^a\frac{(-r)^a}{a!}.
\end{equation*}

We will plug this information into the process we illustrated for the  Nazarov-Tur\'{a}n, and the Borwein-Erd\'{e}lyi bounds to prove Theorem 5. 

\begin{proof}[Proof of Theorem 5]

Taking $r<r_A$, 
\begin{equation}\label{pte6}
	\|\chi_A^{(k)}\|_{l^1(\Z)} \geq 2^k(1-\pi^2r^2)^k\big\|   \widehat{\chi_A} \big\|_{L^{\infty}((\frac{1}{2}-r,\frac{1}{2}+r))}\geq 2^k\frac{2^{a-1}\pi^a}{a!}(1-\pi^2r^2)^kr^a.
\end{equation}
We may assume $a>0$, as otherwise the picking $r=0$ above suffices. This time we   maximize $(1-\pi^2r^2)^k r^{a}$ on $[0,\pi^{-1}].$ The  maximum is attained at \eqref{alp-} with $\alpha=a$. For $k$ large enough this will be less than $r_A.$   The maximum   is \eqref{alp} with $\alpha=a$.
So the optimal value of the  constant above is 
\begin{equation}\label{pte8}
	\begin{aligned}
		=2^k\frac{2^{a-1}}{a!}\Big[\frac{2k}{2k+a}   \Big]^k  \Big[\frac{a}{2k+a}   \Big]^{\frac{a}{2}}.
	\end{aligned}
\end{equation}
Applying the limits we computed in proof of Theorem 4,    \eqref{pte8} satisfies the following asymptotics, that yield  \eqref{bek1} when combined with \eqref{hol}: 
\begin{equation}\label{pte9}
	\begin{aligned}
		\sim 2^k\frac{2^{a-1}}{a!}e^{-\frac{a}{2}}2^{-\frac{a}{2}\log_22k}a^{\frac{a}{2}} =2^{k-\frac{a}{2}\log_22k+a-1} \Big[\frac{a}{e}\Big]^{\frac{a}{2}}\frac{1}{a!}.
	\end{aligned}
\end{equation}

\end{proof}

For certain classes of polynomials the bound $\log_2|A|$ can be replaced by $\log_2|\partial A|$. Let us consider a set $A$ that is composed of intervals of equal length, that is 
\begin{equation*}
	A=\bigcup_{j=1}^n \big\{ a_j,a_j+1,\ldots,a_j+b \big\},
\end{equation*}
where $a_j+b<a_{j+1}.$ Then $A_*$ will shift this set so that $a_1$ will be shifted to $0.$ Let $a_j'$ be shifted $a_j$. Then
\begin{equation*}
	P_A(z)= \sum_{j=1}^n  z^{a_j'}+z^{a'_j+1}+\ldots+z^{a'_j+b} =\big[  1+z+z^2+\ldots +z^b\big]\Big[\sum_{j=1}^n  z^{a_j'}\Big].
\end{equation*}
The first factor in this product does not vanish at $-1$ if $b$ is even, and it vanishes only to first order if $b$ is odd.  The second factor can contain at most   $\log_2 n$  copies of $(z+1).$ So the polynomial $P_A$ can contain at most $1+\log_2n$ copies of $z+1$. But $n=\partial_lA=\partial_rA$. We can then insert this bound into  \eqref{pte9} to essentially replace $\log_2 |A|$ by $\log_2 |\partial A|.$  Whether this is possible for any set $A$ is a very interesting albeit  challenging problem.

\subsection{Sparse sets}

We prove Theorem 6, which gives a uniform result for sparse sets. The method of proof is different from the last three theorems we proved in this section.

\begin{proof}
	We consider 
		\begin{equation*}\label{}
		\begin{aligned}
			\Big\| [e^{2\pi i  x}-1]^k  \sum_{n \in A}  e^{-2\pi i n x} \Big\|_{L^2((\frac{1}{4},\frac{3}{4}))} \geq 2^{k/2}	\Big\|  \sum_{n \in A}  e^{-2\pi i n x} \Big\|_{L^2((\frac{1}{4},\frac{3}{4}))}.
		\end{aligned}
	\end{equation*} 
But 
	\begin{equation*}\label{}
		\begin{aligned}
		\Big\|  \sum_{n \in A}  e^{-2\pi i n x} \Big\|_{L^2((\frac{1}{4},\frac{3}{4}))}^2=\int_{\frac{1}{4}}^{\frac{3}{4}}\Big[|A|+\sum_{\substack{m,n\in A \\ m\neq n}}{e^{2\pi i x (n-m)}}   \Big] dx&=\frac{|A|}{2}+\int_{\frac{1}{4}}^{\frac{3}{4}}\sum_{\substack{m,n\in A \\ m<n }}{2\cos 2\pi(n-m)x}   dx    \\  &   =\frac{|A|}{2}+\sum_{\substack{m,n\in A \\ m<n }}\frac{\sin 2\pi(n-m)x}{\pi (n-m)}\Big|_{\frac{1}{4}}^{\frac{3}{4}}.     
		\end{aligned}
	\end{equation*}
	Now applying our assumption in the theorem we get 
	\begin{equation*}\label{}
		\begin{aligned}
		   \geq \frac{|A|}{2}-\sum_{\substack{m,n\in A \\ m<n }}\frac{2}{\pi (n-m)}\geq \frac{|A|}{2}- \frac{|A|}{4}=\frac{|A|}{4}.    
		\end{aligned}
	\end{equation*}
	Then we have for $p\geq 2$, as in \eqref{hol},

	\begin{equation*}
		\begin{aligned}
			\big\|\chi_A^{(k)}\big \|_{l^{p}(\Z)}\geq 	\big[|A|(k+1)\big]^{\frac{1}{p}-\frac{1}{2}}	\big\|\chi_A^{(k)}\big \|_{l^{2}(\Z)}&= 	\big[|A|(k+1)\big]^{\frac{1}{p}-\frac{1}{2}}	\big\|\widehat{\chi_A^{(k)}}\big \|_{L^{2}(\Tt)}\\  &\geq 	\big[|A|(k+1)\big]^{\frac{1}{p}-\frac{1}{2}} 	\Big\| [e^{2\pi i  x}-1]^k  \sum_{n \in A}  e^{-2\pi i n x} \Big\|_{L^2((\frac{1}{4},\frac{3}{4}))} \\ & \geq\big[|A|(k+1)\big]^{\frac{1}{p}-\frac{1}{2}}  2^{\frac{k}{2}-1}|A|^{\frac{1}{2}}	\\ &= (k+1)^{\frac{1}{p}-\frac{1}{2}}  2^{\frac{k}{2}-1}	\big\|\chi_A\big \|_{l^{p}(\Z)}.
		\end{aligned}
	\end{equation*}

For $1\leq p <2$ we observe that for any $f\in l^{\infty}(\Z) $ 
\begin{equation*}
		\big\|f\big \|_{l^{2}(\Z)}\leq \big\|f\big \|_{l^{p}(\Z)}^{\frac{p}{2}}\big\|f\big \|_{l^{\infty}(\Z)}^{1-\frac{p}{2}}.
\end{equation*}
	Utilizing this we finish the proof:
		\begin{equation*}
		\begin{aligned}
			\big\|\chi_A^{(k)}\big \|_{l^{p}(\Z)}\geq 	\big\|\chi_A^{(k)}\big \|_{l^{\infty}(\Z)}^{1-\frac{2}{p}}	\big\|\chi_A^{(k)}\big \|_{l^{2}(\Z)}^{\frac{2}{p}}\geq 	2^{k\big(1-\frac{2}{p}\big)}	\big\|\widehat{\chi_A^{(k)}}\big \|_{L^{2}(\Tt)}^{\frac{2}{p}}	&\geq 	2^{k\big(1-\frac{2}{p}\big)} 2^{\big(\frac{k-2}{p}\big) }|A|^{\frac{1}{p}}\\	&= 2^{k-\frac{k+2}{p}}	\big\|\chi_A\big \|_{l^{p}(\Z)}.
		\end{aligned}
	\end{equation*}

\end{proof}

\section{Acknowledgements}
The results in this article will constitute a part of the second author's PhD thesis.

\end{document}